\documentclass[reqno,a4paper,11pt]{amsart}
\usepackage{hyperref}
\usepackage{tikz}
\usepackage[active]{srcltx}
\usepackage[curve]{xypic}
\usepackage{graphicx}
\usepackage{mathrsfs}
\usepackage[utf8]{inputenc}
\usepackage{kantlipsum}

\usepackage{color}

\usepackage[active]{srcltx}
\usepackage{epsfig,amsmath}
\usepackage[english]{babel}
\usepackage{amsmath,amsfonts,amssymb,amscd,color,epsfig,amsthm}
\newtheorem{newthm}{Theorem}
\newtheorem{theorem}{Theorem}[section]
\newtheorem{lemma}[theorem]{Lemma}
\newtheorem{proposition}[theorem]{Proposition}
\newtheorem{corollary}[theorem]{Corollary}

\theoremstyle{remark}

\theoremstyle{plain}

\numberwithin{equation}{section}

\newcommand{\wh}{\widehat}

\def\FFF{{\cal F}}

\def\RRR{{\cal R}}

\def\VVV{{\cal V}}

\def\c{\mathbf c}
\def\v{\mathbf v}

\def\g{\gamma}

\def\R{\mbox{$\mathbb R$}}

\def\C{\mbox{$\mathbb C$}}

\def\D{\mathbb D}

\def\lv{ \left(\begin{matrix} }
 \def\rv{\end{matrix}\right)}

\def\cal{\mathcal}

\def\dw{{\dw}}

\newcommand{\mylabel}[1]{\label{#1}}

\newcommand{\REFEQN}[1] { \begin{equation}\mylabel{#1} }
\newcommand{\ENDEQN}{\end{equation}}
\newcommand{\REFTHM}[1] { \begin{theorem}\mylabel{#1} }
\newcommand{\ENDTHM}{\end{theorem}}
\newcommand{\REFNTH}[1] { \begin{newthm}\mylabel{#1} }
\newcommand{\ENDNTH}{\end{newthm}}
\newcommand{\REFPROP}[1]{\begin{proposition}\mylabel{#1} }
\newcommand{\ENDPROP}{\end{proposition} }
\newcommand{\REFLEM}[1]{\begin{lemma}\mylabel{#1} }
\newcommand{\ENDLEM}{\end{lemma} }
\newcommand{\REFCOR}[1]{\begin{corollary}\mylabel{#1} }
\newcommand{\ENDCOR}{\end{corollary} }

\def\pf{postcritically-finite }

\def\ov{\overline}

\usepackage{tikz}
\usetikzlibrary{arrows}
\tikzstyle{every picture}=[> = to]
\tikzset{cdlabel/.style={execute at begin node=$\scriptstyle,execute at end node=$}}
\tikzset{implication/.style={double equal sign distance, -implies}}
\tikzset{biimplication/.style={double equal sign distance, implies-implies}}

\title{Monotonicity of entropy for unimodal real quadratic rational maps}
\author{Yan Gao}
\address{Yan Gao, Mathematical School of Sichuan University, 610065, P.R.China.}
\email{gyan@scu.edu.cn}

\date{}
\begin{document}
\maketitle
\begin{abstract}
We show that the topological entropy is monotonic for unimodal interval maps which are obtained from the restriction of quadratic rational maps with real coefficients. This is done by ruling out the existence of certain post-critical curves in the moduli space of aforementioned maps, and confirms a conjecture made in \cite{2019arXiv190103458F} based on experimental evidence.

\end{abstract}
\section{Introduction}
As one of the most important invariants in dynamics, the variation of topological entropy could help understanding the bifurcations in a family of dynamical systems. In the case of polynomial interval maps with real critical points,  \textit{monotonicity of entropy} conjecture, due to Milnor, asserts that the entropy level sets (the \textit{isentropes}) are connected \cite{MR3289915}. This was established first for quadratic polynomials \cite{MR762431,MR970571,MR1351519}, then in the cubic case  \cite{MR1351522,MR1736945}, and finally, in its full generality, in \cite{MR3264762} (and later in \cite{MR3999686} with a different method).

Similarly, given a real rational map $f\in\Bbb{R}(z)$, it is possible to study its \textit{real entropy}
\begin{equation}\label{real entropy}
h_\Bbb{R}(f):=h_{\rm{top}}\left(f|_{\hat{\Bbb{R}}}=\hat{\Bbb{R}}\rightarrow\hat{\Bbb{R}}\right)
\end{equation}
as $f$ varies in families. With some effort, one can define a \textit{real entropy function} $h_\Bbb{R}$ on an appropriate moduli space of real rational maps
\cite{2018arXiv180304082F}. The domain becomes much more complicated compared to the aforementioned setting of polynomial interval maps with real critical points whose space of affine conjugacy classes has a natural parametrization in terms of critical values \cite[Theorem 3.2]{MR1736945}. For $d=2$, the space ${\rm rat}_2^{\tiny{\R}}$ of Möbius  conjugacy classes of real quadratic rational maps is thoroughly studied by Milnor \cite{MR1246482}: It can be naturally identified with $\Bbb{R}^2$; and admits a partition  to \textit{degree} $\pm 2$, \textit{monotonic}, \textit{unimodal} and \textit{bimodal} regions based on the topological type of the map under consideration. By generating entropy contour plots, paper \cite{2019arXiv190103458F} conjectures that the isentropes are connected in the region of unimodal maps and the region of bimodal maps of shape $(-+-)$ \cite[Conjecture  5.1]{2019arXiv190103458F}, whereas they become disconnected in the region of bimodal maps of shape $(+-+)$ \cite[Conjecture 1.4]{2019arXiv190103458F}.
The latter conjecture is the subject of \cite{2019arXiv191004705F} while the former is partially resolved in  the former paper: the function $h_\Bbb{R}$ is monotonic on the part of the moduli space where maps have real critical points and admit three real fixed points; this region consists of the entirety of the $(-+-)$-bimodal region and a half of the unimodal region \cite[Theorem 1.2]{2019arXiv190103458F}. The main goal of this article is to finish the proof of monotonicity in the unimodal region.

\begin{theorem}\label{thm:quadratic}
There are no bone-loops in the unimodal region of the moduli space of real quadratic rational maps. As a consequence, the entropy function is monotone in the unimodal region.
\end{theorem}
Here, \textit{bones} are connected components of curves defined by post-critical relations $f^n(c)=c$ in the unimodal region of the moduli space. Similar curves also appear in the treatment of the monotonicity problem for real cubic polynomials \cite{MR1351522,MR1736945}. To prove Theorem \ref{thm:quadratic}, we shall utilize the positive transversality result developed in \cite{MR4043216,MR4115082} to argue that no bone can be a Jordan curve (a \textit{bone-loop}). The monotonicity then follows from \cite[Proposition 6.8]{2019arXiv190103458F}.

\noindent{\bf Acknowledgment.} I truly appreciate the contribution of Khashayar Filom to this paper. He gives  many valuable suggestions and comments during the communication, and generously provides the pictures. Both of the two pictures in this paper  come from \cite{2019arXiv191004705F}.

\section{Positive transversality for real holomorphic families}
In this section, we summarize the positive transversality result developed in \cite{MR4115082}. Our statement is slightly different from that in \cite{MR4115082} because we want to apply their theory to a particular setting.

\subsection{Real holomorphic family parameterized by critical values}\label{sec:holomorphic-family}
We say that a map $G:\VVV\times U\to\C$ is a \emph{holomorphic family parameterized by critical values} if
\begin{enumerate}
\item $\VVV$ is an open connected subset of $\C^r$ and $U$ is an open subset of $\C$;
\item $G:(\v,z)\in \VVV\times U\mapsto G_{\v}(z)\in \C$  is a holomorphic map such that  $G_\v$ has $r$ distinct critical points for every $\v\in\VVV$;
\item each $\v=(v_1,\ldots,v_r)\in\VVV$ is an ordered collection of critical values of $G_\v$, i.e.,
\begin{equation}\label{eq:critical-value}
\v=(v_1,\ldots,v_r)=(G_\v(c_1),\ldots,G_\v(c_r)),
\end{equation}
where $c_1,\ldots,c_r$ are all (distinct) critical points of $G_\v$.
\item the ordered collection $\c=(c_1,\ldots,c_r)$ of critical points of $G_\v, \v\in\VVV$, satisfies $G_\v(\c(\v))=\v$ and depends holomorphically on $\v$;
\end{enumerate}
Furthermore, we call this holomorphic family $G$ \emph{real} if for any  $z\in U$ and $\v\in\VVV$, we have $\ov{z}\in U$, $\ov{\v}:=(\ov{v_1},\ldots,\ov{v_r})\in \VVV$, $\c(\ov{\v})=\ov{\c(\v)}$ and $G_{\ov{\v}}(\ov{z})=\ov{G_{\v}(z)}$.

\subsection{Positive transversal unfolding of critical relations}\label{sec:positive-transversal}
Let us fix a real holomorphic family $G:\VVV\times U\to \C$ parameterized by critical values as above. We call $g$ a \emph{real \pf (PCF) map} in $G_\VVV$ if
\begin{itemize}
\item there exists a $\v_0\in\VVV\cap\R^r$ such that $g=G_{\v_0}$;
\item the orbit of each critical point of $g$ stays in $U$ and is finite.
\end{itemize}
By the definition of $G$, the orbits of all critical points of $g$ lie on $\R$. We denote $$P(g):=\{g^k(c):g'(c)=0\text{ and }k\geq1\}.$$

 Since $g$ is PCF, for each $j=1,\ldots,r$, exactly one of the following critical orbit
relations holds:
\begin{itemize}
\item there exists an integer $q_j>0$ and and $\mu(j)\in\{1,\ldots,r\}$ such that $g^{q_j}(c_j(g))=c_{\mu(j)}(g)$ and $g^k(c_j(g))\not\in{\rm crit}(g)$ for each $1\leq k<q_j$;
\item there exist integers $\ell_j<q_j$ such that $g^{q_j}(c_j(g))=g^{\ell_j}(c_j(g))$ and $g^k(c_j(g))\not\in{\rm crit}(g)$ for each $1\leq k<q_j$.
\end{itemize}
Define the map
\[\RRR:=(R_1,\ldots, R_r)\]
from a neighborhood of $\v_0=\v(g)\in\VVV$ into $\C^r$ such that
\begin{equation}\label{eq:first}
R_j:=G^{q_j-1}_\v(v_j)-c_{\mu(j)}(\v)
\end{equation}
if $j$ satisfies the first critical orbit relation, and
\begin{equation}\label{eq:second}
R_j:=G^{q_j-1}_\v(v_j)-G^{\ell_j-1}_\v(v_j)
\end{equation}
if $j$ satisfies the second critical orbit relation, where $\v=(v_1,\ldots,v_r)$.

We say that the family $G_\VVV$ satisfies the \emph{positive transversal} property at $g$ if
\begin{equation}\label{eq:transversal}
\frac{{\rm det}\large(D\RRR(\v(g))\large)}{\prod_{j=1}^r(g^{q_j-1})'(v_j(g))}>0,
\end{equation}
where $D\RRR(\v(g))$ denote the  Jacobian matrix of $\RRR$ at $g$.

In \cite{MR4115082}, the authors investigate a condition, called `the lifting property', under which the positive transversality property holds.

\subsection{Positive transversality along curves with one free critical point}\label{sec:monotone}
Let $G:\VVV\times U\to\C$ be an $r$-dimensional real holomorphic family parameterized by its critical values. Suppose that for each $j\in\{1,\ldots,r-1\}$, there exists a positive interger $q_j$ such that $G_\v^k(c_j)\in U\setminus {\rm crit}(G_\v)$ for $k=1,\ldots,q_j-1$ and $\v\in\VVV$, where each of $c_1,\ldots,c_{r-1}$ is a critical point of $G_\v$.

Choose $1\leq r_1\leq r-1$. With each $1\leq j\leq r_1$ associate  an integer $\mu(j)\in\{1,\ldots,r\}$, and with each $r_1<j\le r$ associate a pair of integers $1\leq\ell_j<q_j$.  Then the  triples $(j,q_j,\mu(j))$ and $(j,q_j,\ell_j)$ induce $r-1$ critical relations, such as $G_\v^{q_j}(c_j)=c_{\mu(j)}$ if $1\leq j\leq r_1$ and $G^{\ell_j}_\v(c_j)=G^{q_j}_\v(c_j)$ if $r_1<j\leq r-1$.

Define a holomorphic map $\RRR_{r-1}:=(R_1,\ldots,R_{r-1}):\VVV\to \C^{r-1}$ by setting
$$R_j(\v):=G_\v^{q_j-1}(v_j)-c_{\mu(j)}(\v)$$
if $1\leq j\leq r_1$, and
\[R_j(\v):=G_\v^{q_j-1}(v_j)-G_\v^{\ell_j-1}(v_j)\]
if $r_1<j\leq r-1$, where $\v=(v_1,\ldots,v_r)\in \VVV$.
Let $L$ be a connected component of the real locus of the zero set
 $$\{\v\in\VVV:\RRR_{r-1}(\v)=\mathbf{0}\}$$
By \cite[Theorem 3.2]{MR4043216}, $L$ is a smooth  curve, i.e., the $r\times(r-1)$ matrix
\begin{equation}\label{eq:matrix}
A_\v:=\left[\frac{1}{(G_\v^{q_1-1})'(v_1)}\nabla R_1(\v),\ldots, \frac{1}{(G_\v^{q_{r-1}-1})'(v_{r-1})}\nabla R_{r-1}(\v)\right]
\end{equation}
has rank $r-1$ for every $\v\in L$.

For any $\v\in L$, the map $G_\v$ has a unique free critical point $c_r\in U$. Let $\v_0\in L$ such that $g:=G_{\v_0}$  is PCF in $G_\VVV$. We define a map
\[R_r^g:=G^{q_r-1}_\v(v_r)-c_{\mu(r)}(\v)\quad\text{or}\quad R_r^g:=G^{q_r-1}_\v(v_r)-G^{\ell_r-1}_\v(v_r)\]
on $\VVV$ according to the critical orbit relation of $c_r(g)$.

\begin{proposition}[\cite{MR4115082},Theorem 8.1]\label{pro:monotone-on-curve}
For each $\v\in L$, define $E_\v$ to be the unique unit vector in $\R^r$ which is
orthogonal to the range of the matrix $A_\v$ defined in \eqref{eq:matrix} and satisfies
${\rm det}\left(A_\v,E_\v\right)>0$. Then
\begin{enumerate}
\item $E_\v$ is a tangent vector to $L$ at $\v$ and $L\ni \v\mapsto E_\v$ is real analytic. In particular, $E_\v$
defines an orientation on the entire curve $L$ which we will call the \emph{positive direction}.
\item Let $g\in L$ be a PCF map in $G_\VVV$. If $G_\VVV$ has the positive transversality condition \eqref{eq:transversal} at $g$, then
\[\frac{1}{(g^{q_r-1})'(v_r(g))}\nabla _{E_g} R_r(g)>0,\]
where  $\nabla _{E_g} R_r(g)$ is the derivative in the direction of the tangent vector $E_g$ of $L$ at $g$.
\end{enumerate}
\end{proposition}

\subsection{Transversality for real quadratic maps}\label{ex:1}
We consider a normalized quadratic rational family
\[\FFF_{(1,0,\infty)}=\FFF_{Z_*}:=\{a(z+1/z)+b, a\in \C\setminus\{0\},{\rm Re}(b)>0\}.\]
It is easy to check that each map  $f\in \FFF_{Z_*}$ satisfies $f(\infty)=\infty,f(0)=\infty$ and $f'( 1)=0$. Moreover,  different maps in $\FFF_{Z_*}$ are not conformally conjugate. On the other hand, any quadratic rational map without super-attracting fixed points is conformally conjugate to either a map in $\FFF_{Z_*}$ or a map in the family $\{a(z+1/z):a\in\C\setminus\{0\}\}$. The latter family determines the \textit{symmetry locus} in ${\rm rat}_2$: any quadratic rational map with a non-trivial Möbius automorphism is conformally conjugate to a member of this family \cite{MR1246482}. For a map $f(z)=a(z+1/z)+b\in\FFF_{Z_*}$, a simple calculation yields the critical values as
$$\left\{
    \begin{array}{ll}
     v_1=f(-1)=-2a+b,  \\
     v_2=f(1)=2a+b.
    \end{array}
  \right.
$$
It follows that $a=(v_2-v_1)/4, b=(v_1+v_2)/2$; hence the family $\FFF_{Z_*}$ can be parameterized by critical values.

More generally, given any three distinct ordered points $Z=(z_1,z_2,z_3)$ in $\ov{\C}$, we have the normalized family
\[\FFF_{Z}:=\{\beta^{-1}_Z\circ f\circ \beta_Z:f\in \FFF_{Z_*}\},\]
where $\beta_Z$ denotes the unique M\"{o}bius transformation sends $(z_1,z_2,z_3)$ to $(1,0,\infty)$. Then a quadratic rational map $f$ belongs to $\FFF_{Z}$ if and only if it is not in the symmetry locus of ${\rm Rat}_2$ and satisfies that $f(z_3)=z_3,f(z_2)=z_3$ and $z_1$ is a critical point of $f$. Since $\FFF_{Z_*}$, hence all $\FFF_Z$, can be parameterized by critical values,
we usually write $f\in\FFF_{Z}$ as $f_{\v}$ with $\v=(v_1,v_2)$.

The transversality of degree-$d$ rational family at its PCF parameter has been studied  in \cite{MR4115082}. By a similar argument as that in \cite[Proposition C.1]{MR4115082}, we obtain the following result. Its proof is postponed to Appendix.

Suppose the triple $Z=(z_1,z_2,z_3)$ satisfies $\infty\in Z\subset\wh{\R}$.  Let $f\in \FFF_{Z}$ be a PCF map with real coefficients such that  $Z\cap P(f)=\emptyset$. We choose a small neighborhood $U$ of $P(f)$ symmetric with respect to $\R$,  and a small neighborhood $\VVV\subseteq \FFF_{Z}$ of $f$ such that $\v\in\VVV$ if and only if $\ov{\v}\in\VVV$. Then
\[G:(\v,z)\in\VVV\times U\mapsto f_\v(z)\in\C\]
is a real holomorphic family parameterized by critical values, and $G_{\v_0}:=f|_{U}$ is a real PCF map in $G_\VVV$. The map $\RRR$ is defined as in Section \ref{sec:positive-transversal} by the critical relation of $f$.

\begin{proposition}\label{pro:quadratic-family}
The real holomorphic family $G_\VVV$ defined above satisfies the positive transversality condition \eqref{eq:transversal}  at $f$.
\end{proposition}

\section{Monotonicity of entropy in the unimodal region}

The moduli space ${\rm rat}_2:={\rm Rat}_2/{\rm PSL}_2$ of quadratic rational maps is defined as the space of conformal conjugacy classes of these maps. A rational map $f\in {\rm Rat}_2$ admits three fixed points counted with multiplicity whose multipliers $\mu_1,\mu_2,\mu_3$ satisfy the \emph{fixed point formula}
\[\frac{1}{1-\mu_1}+\frac{1}{1-\mu_2}+\frac{1}{1-\mu_3}=1.\]
By Milnor \cite{MR1246482}, the moduli space ${\rm rat}_2$ can be parameterized by the pairs
$(\sigma_1,\sigma_2)\in\C^2$ with
\[\sigma_1=\mu_1+\mu_2+\mu_3,\quad \sigma_2=\mu_1\mu_2+\mu_2\mu_3+\mu_3\mu_1.\]

The \emph{real moduli space} ${\rm rat}_2(\R)\subseteq {\rm rat}_2$ is defined as the subset of conjugacy classes of quadratic rational maps with real coefficients. It can be shown that
${\rm rat}_2(\R)=\{(\sigma_1,\sigma_2):\sigma_1,\sigma_2\in\R\}$,
and away from the symmetry locus mentioned above, two maps with real coefficients from the same conjugacy class are conjugate via a real Möbius transformation; so the real entropy \eqref{real entropy} descends to a well defined function on the real moduli space \cite{2019arXiv190103458F}.

Let $f\in {\rm Rat}_2$ have real coefficients and two real critical points. Then $f(\hat{\R})$ is an interval bounded by the two critical values, and thus $h_{\tiny{\R}}(f)$ equals to the topological entropy of interval map $f:f(\hat{\R})\to f(\hat{\R})$. We say that $[f]$ belongs to the \emph{unimodal region} of ${\rm rat}_2^{\tiny{\R}}$ if exactly one critical point, called the \emph{essential} critical point, belongs to the interval $f(\hat{\R})$ and the other one, called the \emph{trivial} critical point,
is outside. The region is confined between the polynomial line $f(c_2)=c_2$ $\sigma_1=2$ and the line $f(c_1)=c_2$ $\sigma_1=-6$ \cite[Figure 15]{MR1246482}. Based on numerically generated entropy contour plots, the isentropes appear to be connected in the unimodal region. This is shown in Figure \ref{fig:plot2} which is adapted from \cite{2019arXiv190103458F}:
In the unimodal region, a normal form for the aforementioned interval maps $f:f(\hat{\R})\to f(\hat{\R})$
is given by
\begin{equation}\label{family}
x\mapsto\frac{2\mu x(tx+2)}{\mu^2x^2+(tx+2)^2}:[-1,1]\rightarrow [-1,1]\quad (\mu,t<0,\, t-2< \mu< |t+2|).
\end{equation}
Entropy contour plots in the $(\mu,t)$-plane project to the entropy contour plots in the unimodal region of
${\rm rat}_2^{\tiny{\R}}$ by a finite-to-one map (see \cite[Section 4]{2019arXiv190103458F} for details).

\begin{figure}[ht!]
\center
\includegraphics[width=9.5cm]{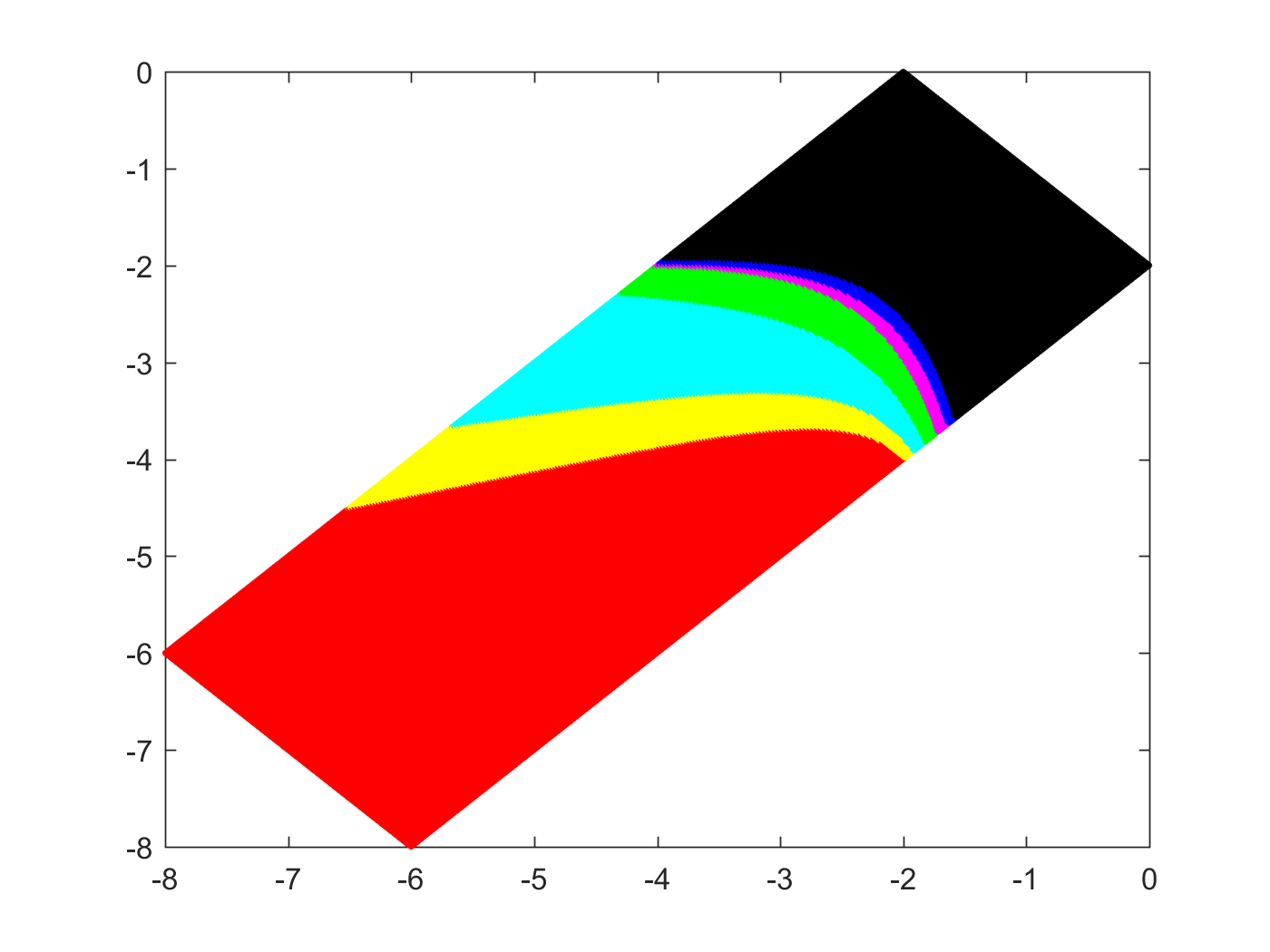}
\caption{An entropy contour plot for the unimodal quadratic rational maps \eqref{family} adapted from \cite{2019arXiv190103458F}. The ordering of colors is black$<$blue$<$magenta$<$green$<$cyan$<$yellow$<$red and they  correspond to the partition $[0, 0.1)$, $[0.1, 0.25)$, $[0.25, 0.4)$, $[0.4, 0.48)$, $[0.48, 0.55)$, $[0.55, 0.65)$ and $[0.65, \log(2)]$ of  $[0,\log(2)\approx 0.7]$. The connectedness of entropy level sets here suggests the monotonicity of entropy in the unimodal region of ${\rm rat}_2(\R)$.}
\label{fig:plot2}
\end{figure}

\begin{figure}[ht!]
\includegraphics[width=15cm, height=9cm]{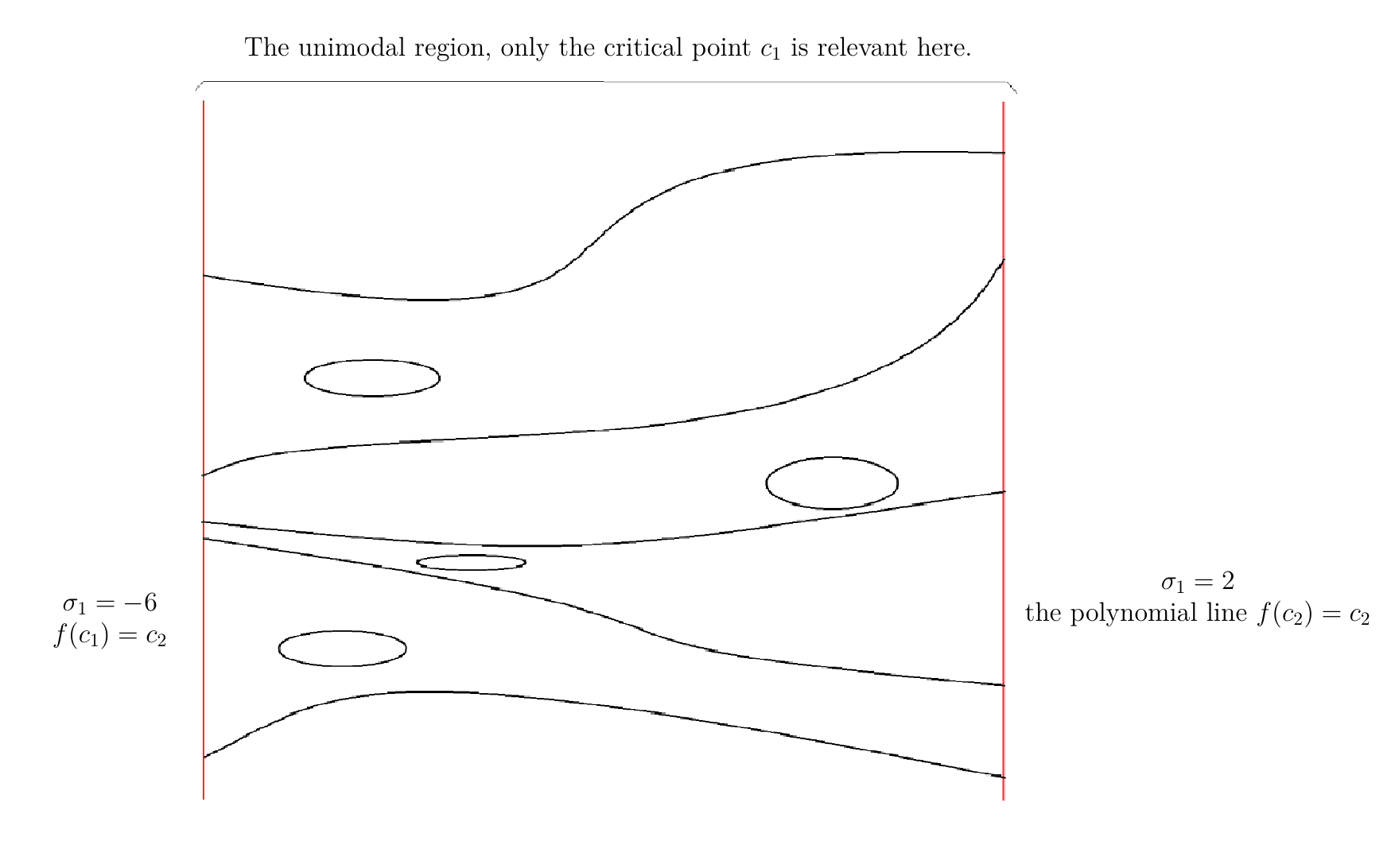}
\caption{A schematic picture of bones (in black)  defined by the post-critical relations $f^{\circ n}(c_1)=c_1$ in the unimodal region of the $(\sigma_1,\sigma_2)$-plane.  They are disjoint one-dimensional submanifolds. There are  bone-arcs connecting PCF quadratic polynomials to PCF quadratic rational maps on the $f(c_1)=c_2$ line. Any other bone in this region should be a Jordan curve, a bone-loop. The picture is adapted from \cite{2019arXiv190103458F}.}
\label{fig:bone}
\end{figure}

If $[f]$ belongs to the (open) unimodal region with $f\in {\rm Rat}_2(\R)$, then $f$ is not in the symmetry locus and has no super-attracting fixed point. It follows that the unimodal region is included in the real locus $\FFF_Z(\R)$ of $\FFF_Z$ defined in subsection \ref{ex:1}.  Therefore, we always consider the unimodal region of ${\rm rat}_2(\R)$  in the normalized family $\FFF_Z(\R)$ for a suitable triple $Z\subset \wh{\R}$.

In the unimodal region, the essential/trivial critical point moves continuously. So we always assume that $c_1(f)$ is the essential critical point.  As indicated in \cite{2019arXiv190103458F}, the real locus of the curves $f^{n}(c_1)=c_1$ consists of finite disjoint one-dimensional closed sub-manifolds of the plane ${\rm Rat}_2(\R)\cong \R^2$.
Following \cite{2019arXiv190103458F}, each component of such real locus is called a \emph{bone}, which is either a Jordan curve, called \emph{bone-loop}, or a closed subset of the plane diffeomorphic to the real line, called \emph{bone-arc}; see Figure \ref{fig:bone}.

\begin{proof}[Proof of Theorem \ref{thm:quadratic}]
We prove this theorem in the normalized family $\FFF_Z(\R)$  with $Z=(\infty,0,1)$ such that $c_2(f)=\infty$ is the trivial critical point for any $f$ in the unimodal region. Then, in this region, we have $f(\wh{\R})\subset \R$.

We first claim that if  bone-loops exist, then some bone-loops contain hyperbolic PCF maps. If it is not true, all hyperbolic PCF maps in the unimodal region are on bone-arcs. Let $\g$ be a bone-loop. By the density of hyperbolicity \cite{MR2342693}, there exists a hyperbolic PCF map $g$ in the region bounded by $\g$. Then the bone-arc passing through $g$ has to intersects $\g$, a contradiction. We then complete the proof of the claim.

We will obtain a contradiction by assuming that $\g$ is a bone-loop in the unimodal region containing a hyperbolic PCF map $f_{\v_0}$. Let $m,n\geq1$ be the minimal numbers for which
$$f_{\v_0}^n(c_1(\v_0))=c_1(\v_0) \text{ and } f^m_{\v_0}(c_2(\v_0))=c_1(\v_0).$$
Since the period of $c_1$ is locally constant, then for every $\v\in\g$ we have $f_\v^n(c_1)=c_1$ and $f_\v^k(c_1)\not=c_1$ for $k=1,\ldots,n-1$. By changing $\v_0$ if necessary, one can assume that
\begin{equation}\label{eq:assumption}
\text{for any $\v\in\g$ and $i<m$, $f_\v^i(c_2)\not=c_1$.}
\end{equation}

Define a map $R_1:= f_\v^{n-1}(v_1)-c_1$. The bone-loop $\g$ is the component  of the zero set of $R_1=0$ that contains $\v_0$.
We define a direction on $\g$ by the tangent vector
\[E_\v:=(f_\v^{n-1})'(v_1)\left(-\frac{\partial R_1}{\partial v_2},\frac{\partial R_1}{\partial v_1}\right),\]
called `positive direction' following Proposition \ref{pro:monotone-on-curve}.

Let $\alpha$ be the component containing $v_0$ of the zero set of
\[R_2(\v):=f_{\v}^{m-1}(v_2)-c_1,\]
Then $\alpha$ is a closed real $1$-dimensional submaniford of $\FFF_Z$ (see for example \cite{MR4043216}). Since $\v_0$ belongs to the unimodal region, the postcritical set of $f_{\v_0}$ avoids $Z$. Thus by Proposition \ref{pro:quadratic-family}, we have the positive transversality condition
\begin{equation}\label{eq:positive}
\frac{1}{(f_{\v_0}^{m-1})'(v_2)(f_{\v_0}^{n-1})'(v_1)}{\rm det}
\left(
\begin{array}{cc}
\dfrac{\partial R_1}{\partial v_1}&\dfrac{\partial R_1}{\partial v_2}\\[10pt]
\dfrac{\partial R_2}{\partial v_1}&\dfrac{\partial R_2}{\partial v_2}
\end{array}
\right)_{\v_0}>0
\end{equation}
at $\v_0$. It follows that the curve $\alpha$ transversally cross $\g$ at $\v_0$.
Since $\g$ is a Jordan curve, we have $\#(\alpha\cap\g)\geq2$. Let $\v_1\in \g\cap\alpha$ be the first point we meet when going along $\g$ starting from $\v_0$ in the positive direction.

By Proposition \ref{pro:monotone-on-curve}, the condition \eqref{eq:positive} implies
\begin{equation*}
\frac{1}{(f_{\v_0}^{m-1})'(v_2)}\nabla_{E_{\v_0}}R_2(\v_0)>0,
\end{equation*}
where $\nabla_{E_{\v_0}}$ stands for the directional derivative of $R_2$ in the direction $E_{\v_0}$. To be definite,
let us assume that
\begin{equation}\label{eq:proof-1}
(f_{\v_0}^{m-1})'(v_2)>0.
\end{equation}
This implies that
\begin{equation}\label{eq:proof-2}
\nabla_{E_{\v_0}}R_2(\v_0)>0.
\end{equation}
We now consider the case at $\v_1$. By \eqref{eq:assumption}, $m$ is the minimal number with $f_{\v_1}^m(c_2)=c_1$. Then, similarly as before, we have
\[\frac{1}{(f_{\v_1}^{m-1})'(v_2)}\nabla_{E_{\v_1}}R_2(\v_1)>0.\]
Due to continuity, \eqref{eq:proof-1} implies $(f_{\v_1}^{m-1})'(v_2)>0$ as otherwise there is a parameter $\v$ at which
$(f_{\v}^{m-1})'(v_2)=0$ contradicting \eqref{eq:assumption}. As a consequence, we get
 \begin{equation}\label{eq:proof-3}
\nabla_{E_{\v_1}}R_2(\v_1)>0.
\end{equation}

 Remember the segment $f_\v(\wh{\R})\subseteq\R$ for all $\v\in\g$.  Notice that the order of $c_1,f_\v(c_1),\ldots,f^{n-1}_{\v}(c_1)$ in $f(\R)$ does not change along $\g$. It then follows from \eqref{eq:proof-2} that $f_\v^m(c_2)$ is on the right of $c_1$ for all $\v$ in the subarc $\g(\v_0,\v_1)$ of $\g$ from $\v_0 $ to $\v_1$ in the positive direction. Therefore, in order that $f^m_\v(c_2)$ hits $c_1$ at $\v_1$, the function $R_2(\v)=f^{m-1}_\v(v_2)-c_1$ has to decrease when $\v$ tends to $\v_1$ along $\g(\v_0,\v_1)$. It means $\nabla_{E_{\v_1}}R_2(\v_1)\leq0$, contradicting \eqref{eq:proof-3}.
\end{proof}

\begin{appendix}
\section{The proof of Proposition \ref{pro:quadratic-family}}
To prove the proposition, we need to check the lifting property stated below.
Given a holomorphic motion $h_\lambda$ of $P(g)$ over $(\D,0)$, we say that
$\hat{h}_\lambda$ is a \emph{lift} of $h_\lambda$ over $\D_\epsilon$  with
respect to $(g,G_\VVV)$ if for all $\lambda\in\D_\epsilon$,
\begin{itemize}
\item $\hat{h}_\lambda(c_j(g))=c_j(\v(\lambda))$ for each $j=1,\ldots,r$ and $c_j(g)\in P(g)$, where
\[\v(\lambda)=(h_\lambda(v_1(g)),\ldots,h_\lambda(v_r(g)));\]
\item $G_{\v(\lambda)}(\hat{h}_\lambda(x))=h_\lambda(g(x))$ for each $x\in P(g)\setminus {\rm crit}(g)$.
\end{itemize}
We say that $G_\VVV$ has the \emph{real lifting property at $g$}  if for each real holomorphic motion $h_\lambda^{(0)}$
of $P(g)$ over $(\D,0)$ there exist $\epsilon>0$ and a sequence of real holomorphic motions $h_\lambda^{(k)},k=1,2,\ldots$
of $P(g)$ over $(\D_\epsilon,0)$ such that for each $k\geq0$,
\begin{itemize}
\item $h_{\lambda}^{(k+1)}$ is a lift of $h_{\lambda}^{(k)}$;
\item there exists $M>0$ such that $|h_\lambda^{(k)}(x)|\leq M$ for all $x\in P(g)$ and all $\lambda\in \D_\epsilon$.
\end{itemize}

\begin{theorem}[Main Theorem in \cite{MR4115082}]\label{thm:LSS}
Let $G_\VVV$ be a real holomorphic family parameterized by critical values and $g$ a real PCF map in $G_\VVV$. Assume that $G_\VVV$ is not a Latt\'{e}s family and satisfies the lifting property at $g$, then $G_\VVV$ has the positive transversality property \eqref{eq:transversal} at $g$.
\end{theorem}

\begin{proof}[Proof of Proposition \ref{pro:quadratic-family}]
By Theorem \ref{thm:LSS}, we just need to check that the real holomorphic family $G_\VVV$ defined in Subsection \ref{ex:1} has the lifting property at $f$.

Let $h_\lambda^{(0)}$ be an arbitrary real holomorphic motion of $P=P(f)$ over $(\D,0)$. By Bers-Royden \cite{MR857675},
there exists $\epsilon>0$ such that $h_\lambda^{(0)}$, $\lambda\in\D_\epsilon$, extends to a real holomorphic motion of $\ov{\C}$ over $(\D_\epsilon,0)$, which satisfies the normalization property that $h_\lambda^{(0)}(z_1,z_2,z_3)=(z_1,z_2,z_3)$ (since $z_1,z_2,z_3\not\in P$).
Denote $\mu_\lambda^{0}$ the complex dilatation of $h_\lambda^{0}$.

For $k\geq1$, we inductively define  $\mu_\lambda^{(k)}:=f^{*}\mu_{\lambda}^{(k-1)}$, and $h_\lambda^{(k)}$ the unique quasi-conformal map with complex dilatation $\mu_\lambda^{(k)}$ such that $h_\lambda^{(k)}(z_1,z_2,z_3)=(z_1,z_2,z_3)$.
Then by the Measurable Riemann Mapping Theorem, $h_\lambda^{(k)}$ is a real holomorphic motion of $\ov{\C}$ over $(\D_\epsilon,0)$.

For each $k\geq0$, to prove that $h_{\lambda}^{(k+1)}$, restricting to $P$, is a lift of $h_\lambda^{(k)}$, restricting to $P$, with respect to $G_\VVV$, it is enough to show that
\[h_\lambda^{(k)}\circ f\circ (h_\lambda^{(k+1)})^{-1}=f_{(h_\lambda^{(k)}(v_1),h_\lambda^{(k)}(v_2))},\]
for $|\lambda|$ small enough. To see this, note first that for each $\lambda\in\D_\epsilon$, the map
 \begin{equation}\label{eq:lift-Newton}
f_\lambda:=h_\lambda^{(k)}\circ f\circ (h_\lambda^{(k+1)})^{-1}
\end{equation}
is a rational map of degree $2$. Furthermore, according to the normalization of $h_\lambda^{(k)},h_\lambda^{(k+1)}$, we have that $f_\lambda(z_3)=z_3, f_\lambda(z_2)=z_3, f_\lambda'(z_1)=0$, and then $f_\lambda\in \FFF_{Z}$ if $\epsilon$ is sufficiently small. Since the maps in $\FFF_{Z}$ can be parameterized by  critical values,  it follows that $f_\lambda=f_{(h_\lambda^{(k)}(v_1),h_\lambda^{(k)}(v_2))}$.
\end{proof}
\end{appendix}

\bibliographystyle{alpha}
\bibliography{bibcomprehensive}

\begin{thebibliography}{DGMT95}

\bibitem[BR86]{MR857675}
Lipman Bers and H.~L. Royden.
\newblock Holomorphic families of injections.
\newblock {\em Acta Math.}, 157(3-4):259--286, 1986.

\bibitem[BvS15]{MR3264762}
H.~Bruin and S.~van Strien.
\newblock Monotonicity of entropy for real multimodal maps.
\newblock {\em J. Amer. Math. Soc.}, 28(1):1--61, 2015.

\bibitem[DGMT95]{MR1351522}
S.~P. Dawson, R.~Galeeva, J.~W. Milnor, and C.~Tresser.
\newblock A monotonicity conjecture for real cubic maps.
\newblock In {\em Real and complex dynamical systems ({H}iller\o d, 1993)},
  volume 464 of {\em NATO Adv. Sci. Inst. Ser. C Math. Phys. Sci.}, pages
  165--183. Kluwer Acad. Publ., Dordrecht, 1995.

\bibitem[DH85]{MR762431}
A.~Douady and J.~H. Hubbard.
\newblock {\em \'{E}tude dynamique des polyn\^{o}mes complexes. {I},II}.
\newblock Publications Math\'{e}matiques d'Orsay. 1984,1985.

\bibitem[Dou95]{MR1351519}
A.~Douady.
\newblock Topological entropy of unimodal maps: monotonicity for quadratic
  polynomials.
\newblock In {\em Real and complex dynamical systems ({H}iller\o d, 1993)},
  volume 464 of {\em NATO Adv. Sci. Inst. Ser. C Math. Phys. Sci.}, pages
  65--87. Kluwer Acad. Publ., Dordrecht, 1995.

\bibitem[{Fil}18]{2018arXiv180304082F}
K.~{Filom}.
\newblock {Real entropy rigidity under quasi-conformal deformations}.
\newblock {\em arXiv e-prints}, page arXiv:1803.04082v2, Mar 2018.

\bibitem[{Fil}19]{2019arXiv190103458F}
K.~{Filom}.
\newblock {Monotonicity of entropy for real quadratic rational maps}.
\newblock {\em arXiv e-prints}, page arXiv:1901.03458v2, Jan 2019.

\bibitem[FP19]{2019arXiv191004705F}
K.~{Filom} and K.~M. {Pilgrim}.
\newblock {On the non-monotonicity of entropy for a class of real quadratic
  rational maps}.
\newblock {\em arXiv e-prints, to appear in the Journal of Modern Dynamics},
  page arXiv:1910.04705v2, Oct 2019.

\bibitem[Koz19]{MR3999686}
O.~Kozlovski.
\newblock On the structure of isentropes of real polynomials.
\newblock {\em J. Lond. Math. Soc. (2)}, 100(1):159--182, 2019.

\bibitem[KSvS07]{MR2342693}
O.~Kozlovski, W.~Shen, and S.~van Strien.
\newblock Density of hyperbolicity in dimension one.
\newblock {\em Ann. of Math. (2)}, 166(1):145--182, 2007.

\bibitem[LSvS19]{MR4043216}
Genadi Levin, Weixiao Shen, and Sebastian van Strien.
\newblock Transversality for critical relations of families of rational maps:
  an elementary proof.
\newblock In {\em New trends in one-dimensional dynamics}, volume 285 of {\em
  Springer Proc. Math. Stat.}, pages 201--220. Springer, Cham, [2019]
  \copyright 2019.

\bibitem[LSvS20]{MR4115082}
Genadi Levin, Weixiao Shen, and Sebastian van Strien.
\newblock Positive transversality via transfer operators and holomorphic
  motions with applications to monotonicity for interval maps.
\newblock {\em Nonlinearity}, 33(8):3970--4012, 2020.

\bibitem[Mil93]{MR1246482}
J.~W. Milnor.
\newblock Geometry and dynamics of quadratic rational maps.
\newblock {\em Experiment. Math.}, 2(1):37--83, 1993.
\newblock With an appendix by the author and Lei Tan.

\bibitem[MT88]{MR970571}
J.~W. Milnor and W.~P. Thurston.
\newblock On iterated maps of the interval.
\newblock In {\em Dynamical systems ({C}ollege {P}ark, {MD}, 1986--87)}, volume
  1342 of {\em Lecture Notes in Math.}, pages 465--563. Springer, Berlin, 1988.

\bibitem[MT00]{MR1736945}
J.~W. Milnor and C.~Tresser.
\newblock On entropy and monotonicity for real cubic maps.
\newblock {\em Comm. Math. Phys.}, 209(1):123--178, 2000.
\newblock With an appendix by Adrien Douady and Pierrette Sentenac.

\bibitem[vS14]{MR3289915}
S.~van Strien.
\newblock Milnor's conjecture on monotonicity of topological entropy: results
  and questions.
\newblock In {\em Frontiers in complex dynamics}, volume~51 of {\em Princeton
  Math. Ser.}, pages 323--337. Princeton Univ. Press, Princeton, NJ, 2014.

\end{thebibliography}

\end{document}